\documentclass[12pt,a4paper, reqno]{amsart}
\usepackage{amsfonts,amsmath,amsthm,amssymb, amscd}
\usepackage{enumerate}

\newcommand*{\gp}[1]{\langle\;#1\;\rangle}
\newcommand*{\order}[1]{\vert #1 \vert}
\newcommand*{\Y}{\mathbin{\mathsf Y}}

\newtheorem{theorem}{Theorem} 

\newtheorem{lemma}{Lemma}
\newtheorem{corollary}{Corollary}

 \DeclareMathOperator{\supp}{supp} 
\DeclareMathOperator{\im}{im} 

\begin{document}

\title[On the Unitary Subgroups of group algebras]{On the Unitary Subgroups of group algebras}

\author{Zsolt Adam Balogh}

\address{Department of Math. Sci, College of Science, United Arab Emirates University, Al Ain, UAE}
\email{baloghzsa@gmail.com}



\begin{abstract}
Let $FG$ be the group algebra of a finite $p$-group $G$ over a finite field $F$ of characteristic $p$ and $*$ the classical involution of $FG$. The $*$-unitary subgroup of $FG$, denoted by $V_*(FG)$, is defined to be the set of all normalized units $u$ satisfying the property $u^*=u^{-1}$.
In this paper we give a recursive method how to compute the order of the $*$-unitary subgroup for many non-commutative group algebras. We also prove a variant of the modular isomorphism question of group algebras, where $F$ is a finite field of characteristic two, that is $V_*(FG)$ determines the basic group $G$ for all non-abelian $2$-groups $G$ of order at most $2^4$.  
\end{abstract}

\maketitle

\section{Introduction}

Let $FG$ be the group algebra of a finite $p$-group $G$ over a finite field $F$ of characteristic $p$. Let $V(FG)$ denotes the group of normalized units in $FG$. The description of the structure of $V(FG)$ is a central problem in the theory of group algebras and it has been investigated by several authors. 
For an excellent survey on group of units of modular group algebras we refer the reader to \cite{Bovdi_survey}.

An element $u \in V(FG)$ is called unitary if $u^*=u^{-1}$, with respect to the
classical $*$-involution of $FG$ (the linear extension of the involution on $G$ which sends each element of $G$ to its inverse). The set of all unitary elements of $V(FG)$ forms a subgroup of $V(FG)$ which is denoted by $V_*(FG)$ and is called $*$-unitary subgroup. This subgroup plays an important role of studying the structure of the group of units of group algebras. The $*$-unitary subgroup has been investigated in several papers (\cite{Bovdi_Erdei_II}, \cite{Bovdi_Erdei_I},  \cite{Bovdi_Szakacs_II}, \cite{Bovdi_Kovacs_I}, \cite{Bovdi_Grichkov}, \cite{Creedon_Gildea_I}, \cite{Creedon_Gildea_II}, \cite{spinelli_2},  \cite{spinelli_1}). 

The order of $*$-unitary subgroup when $G$ is a $p$-group and $p$ is an odd prime is given in  \cite{Bovdi_Szakacs_III} and \cite{Bovdi_Rosa_I}. 
To compute the order of $V_*(FG)$  when $G$ is a $2$ group and $p=2$ is an open and is a particularly challenging problem. It is to be expected that the order is divisible by $|F|^{\frac{1}{2}(\order{G}+\order{G_T})-1}$, where $G_T$ is the set of elements of order two in $G$.
In paper \cite{Bovdi_Szakacs_III} the authors confirmed this conjecture for abelian $2$-groups and finite fields of characteristic $2$. The conjecture was confirmed for dihedral and generalized quaternion groups in \cite{Bovdi_Rosa_I}.
In the second and third sections we give a recursive method how to compute the order of the $*$-unitary subgroups and confirm the conjecture for some non-abelian $2$-group classes.   

The modular isomorphism problem is an old and unanswered problem in the theory of group representation. A stronger variant of the problem is said to be the \emph{isomorphism problem of normalized units} (UIP) is due to Berman \cite{Berman_I}. Let $F$ be a finite field of characteristic $p$, $G$ and $H$ be finite $p$-groups such that $V(FG)$ and $V(FH)$ are isomorphic. 
One may ask whether $G$ and $H$ are isomorphic groups? 
The studies in \cite{Balogh_Bovdi_I}, \cite{Balogh_Bovdi_II} and \cite{konovalov} resulted in proving the conjecture for some group classes.
The $*$-unitary group of a group algebra is a small subgroup in $V(FG)$ so it is interesting to ask whether this smaller subgroup determines the basic group $G$ or not.
This problem is called the \emph{$*$-unitary isomorphism problem} (*-UIP).
In the last section we prove that $V_*(FG)\cong V_*(FH)$ implies that $G\cong H$ for all non-abelian groups $G$ and $H$ of order at most $2^4$, where $F$ is any finite field of characteristic two.

\section{On the order of unitary subgroups}

Let $G$ be a finite $2$-group. We will denote by $G[2^i]$ the subgroup of $G$ generated by the elements of order $2^i$.
We use the notation $G^{2^i}$ for the subgroup $\gp{ g^{2^i} \,\vert\,g \in G }$.  Let $\zeta(G)$ be the center and
$G'$ the commutator subgroup of $G$, respectively. 
Let $\supp(x)$ denote the support of $x\in FG$ and
$x^g=g^{-1}xg$, where $g\in G$. We define $\widehat{C}=\sum_{g\in C} g$, where $C$ is a subset of $G$. Throughout this paper $|S|$ denotes the cardinality of the finite set $S$ and $\order{g}$ the order of $g\in G$. Let $G_T$ be the set of elements of order two in $G$, that is $G_T=\{g \in G \;|\; g^2=1 \}$. 

The following two lemma will be useful.
\begin{lemma}(\cite[Theorem 2]{Bovdi_Szakacs_I})\label{szakacs}
	Let $G$ be a finite abelian $2$-group and ${F}$ a finite field of characteristic two.
	Then \[\order{V_*({F}G)}=\order{G^2[2]}\cdot |F|^{\frac{1}{2}(\order{G}+\order{G_T})-1}.\]
\end{lemma}

\begin{lemma}(\cite[Corollary 2]{Bovdi_Rosa_I})\label{roza}
	Let ${F}$ be a finite field of characteristic two.
Then 
\begin{enumerate}
	\item[(i)] $|V_*(FG)|$ equals $|F|^{\frac{|G|+|G_T|}{2}-1}$ if $G$ is a dihedral $2$-group;
	\item[(ii)] $\order{V_*(FG)}$ equals $4\cdot \order{F}^{\frac{|G|+|G_T|}{2}-1}$ if $G$ is a generalized quaternion $2$-group.
\end{enumerate}
\end{lemma}

Let $H$ be a normal subgroup of $G$ and $I(H)$ the ideal of $FG$ generated by the set $\{(1+h) \;|\; h\in H\}$. 
It is well-known that $FG/I(H)\cong F(G/H)$ and we denote by $\Psi$ the corresponding natural homomorphism. Let us denote by $V_*(F\overline{G})$ the unitary subgroup of the factor algebra $FG/I(H)$, where $\overline{G}=G/H$. It is clear that the set 
\[N^*_{\Psi}=\{x\in V(FG) \;|\; \Psi(x) \in V_*(F\overline{G}) \}\] 
forms a subgroup in $V(FG)$. Furthermore, the set $I(H)^+=\{1+x \, |\, x\in I(H) \}$ forms a normal subgroup in $V(FG)$. We define $S_H$ to be the group generated by the elements $\{ xx^* \;|\; x\in N^*_{\Psi} \}$. It is obvious that $S_H$ is a subgroup of $I(H)^+$. Indeed, $xx^* \in 1+\ker(\Psi)=I(H)^+$.

\begin{lemma}\label{lemma_main}
	Let $G$ be a finite $2$-group, $H$ its normal subgroup of order two and $F$ a finite field of characteristic two. If $S_H$ is central in $N^*_{\Psi}$, then the order of $V_*(FG)$ is equal to $\order{F}^{\frac{\order{G}}{2}}\cdot \frac{|V_*(F\overline{G})|}{|S_H|}$.
\end{lemma}

\begin{proof}
	Let $\Phi$ be the mapping defined by $\Phi(x)=xx^*$ for every $x\in V(FG)$. It is easily seen that $\Phi$ is generally not a group homomorphism on $V(FG)$. 
	However, $S_H$ being central in $N^*_{\Psi}$ implies that the restriction $\Phi|_{N^*_{\Psi}}$ is a homomorphism. 
	
	According to the homomorphism theorem $N^*_{\Psi}/\ker(\Phi|_{N^*_{\Psi}})\cong S_H$. Since $\ker(\Phi|_{N^*_{\Psi}})=V_*(FG)$ we have
	\[
	|V_*(FG)|=\frac{|N^*_{\Psi}|}{|S_H|}=\frac{ |I(H)^+|\cdot |V_*(F\overline{G})| }{|S_H|}.
	\]
	
	Evidently, $I(H)$ can be considered as a vector space over $F$ with basis $\{ u(1+h) \;|\; u \in T(G/H),\; h \in H \}$, where $T(G/H)$ is a complete set of left coset representatives of $H$ in $G$. Thus we have that
	$|I(H)^+|=|F|^{\frac{\order{G}}{2}}$ and 
	\[
	|V_*(FG)|=\order{F}^{\frac{\order{G}}{2}}\cdot \frac{|V_*(F\overline{G})|}{|S_H|}.
	\]
	
\end{proof}

	Let $C$ be a central subgroup of a $2$-group $G$, $F$ a finite field of characteristic two and $g_1,\cdots,g_n \in G$ for some $n$. We denote by $V_{g_1,\cdots,g_n}$ the vector space in $FG$ over $F$ spanned by the elements $g_i\widehat{C}$. Let $G_{g_1,\cdots,g_n}$ denote the group generated by the elements $1+\alpha g_i\widehat{C}$, $\alpha \in F$.

\begin{lemma}\label{lemma_6}
	The set $1+V_{g_1,\cdots,g_n}$ coincides with  $G_{g_1,\cdots,g_n}$.
	
\end{lemma}

\begin{proof}
	Let $x_1,x_2\in FG$ be.	The identity $1+(x_1+x_2)\widehat{C}=1+x_1\widehat{C}+x_2\widehat{C}=(1+x_1\widehat{C})(1+x_2\widehat{C})$ proves the lemma.
	
\end{proof}

\begin{lemma}\label{lemma_3}
	Let $G$ be a $2$-group and $F$ a finite field of characteristic two. Then $\supp(xx^*)\cap G_T=\{1\}$ for every $x \in V(FG)$.
\end{lemma}

\begin{proof}
	Let $x=\sum_{i=1}^{\order{G}}\alpha_i g_i \in V(FG)$. Then \[xx^*=\Big(\sum_{i=1}^{\order{G}}\alpha_i g_i\Big)\Big(\sum_{j=1}^{\order{G}}\alpha_j g_j^{-1}\Big)=1+\sum_{1\leq i<j \leq \order{G}} \alpha_i\alpha_j \big(g_ig_j^{-1}+(g_ig_j^{-1})^{-1}\big).\]
	
	Assume that $g_ig_j^{-1}\in G_T$ for some $i$ and $j$. Then $(g_ig_j^{-1}+(g_ig_j^{-1})^{-1})=0$, which proves the lemma. 
	
\end{proof}

Here $H$ denotes a central subgroup of order two in $G$ generated by $c\in G$. Let $M$ be the set $\{g \in G \;|\; g^2=c \}$.

\begin{lemma}\label{lemma_4}
	Let $G$ be a $2$-group and $F$ a finite field of characteristic two. If $1+g\widehat{H} \in S_H$ for some $g\in G$, then $g^2=c$.
\end{lemma}

\begin{proof}
	Assume that $1+g\widehat{H} \in S_H$ for some $g\in G$. Since $S_H$ contains only $*$-symmetric elements $1+g\widehat{H}=1+g^{-1}\widehat{H}$.
	Therefore $(g+g^{-1})\widehat{H}=g+gc+g^{-1}+g^{-1}c=0$. If $g=g^{-1}$, then $\order{g}=2$, which is impossible by Lemma \ref{lemma_3}.
	Thus, $g=g^{-1}c$ and $g^2=c$.
	
\end{proof}

\begin{corollary}\label{corrolary_1}
	Let $G$ be a $2$-group, $H$ its central subgroup of order two and $F$ a finite field of characteristic two. Then $S_H$ can be generated by the elements of the form $1+\alpha_g g\widehat{H}$, where $g\in M$ and $1+\beta_h (h+h^{-1})\widehat{H}$, where $h \not\in M$ $(\alpha_g,\beta_h \in F)$.
\end{corollary}

\begin{proof}
	Obviously, $gh+(gh)^{-1}=gh(1+(gh)^{-2})$.
	We have already proved that $S_H \subseteq I(H)^+$ and $S_H$ contains only $*$-symmetric elements.
	Therefore every $x \in S_H$ can be expressed in the following form
	\[
	x=1+\sum_{g\in M} \alpha_g g\widehat{H} + \sum_{h\not\in M} \beta_h (h+h^{-1})\widehat{H}
	\]
	by Lemma \ref{lemma_3} and \ref{lemma_4}.
	Since $(1+x_1\widehat{H})(1+x_2\widehat{H})=1+(x_1+x_2)\widehat{H}$ for every $x_1,x_2\in FG$ the corollary is done.
\end{proof}

\begin{lemma}\label{lemma_5}
	Let $G$ be a $2$-group, $H$ its central subgroup of order two and $F$ a finite field of characteristic two. Let $g \in G$ be such that $g^2\not\in H$. Then $1+\alpha (g+g^{-1})\widehat{H} \in S_H$.
\end{lemma}

\begin{proof}
	Suppose that $g \in G$ and $g^2\not\in H$. Since $g\ne g^{-1}$ and $1+\alpha g\widehat{H} \in \ker(\Psi)$ we have
	\[
	\begin{split}
	(1+\alpha g\widehat{H})(1+\alpha g\widehat{H})^*&=1+\alpha (g+g^{-1})\widehat{H} \\
	\end{split}
	\]
	for every $\alpha \in F$ which proves the lemma.	
	
\end{proof}

Set $G_P=\{g^2 \;|\; g \in G \}$.
Let $\Theta$ denote the set of all groups with the property that $g^h=g$ or $g^h=g^{-1}$ for all $g\in G\setminus G_T$ and $h\in G$. It is clear that every abelian group belongs to $\Theta$. One can check that the dihedral and generalized quaternion groups belong to $\Theta$.

\begin{theorem}\label{theorem_main}
	Let $G$ be a finite $2$-group and $c\in \zeta(G)[2]\setminus G_P$, $C=\gp{c}$, $\overline{G}\cong G/C$ and $F$ a finite field of characteristic two. If $\overline{G}\in \Theta$, then 
	\[
	\order{V_*({F}G}=\order{F}^{\frac{\order{G}+\order{G_T}}{4}}
	\cdot \order{V_*({F}\overline{G})}.
	\] 
\end{theorem}

\begin{proof}
	Assume that $c\in \zeta(G)[2]\setminus G_P$. Then $\widehat{C}=1+c\ne 0$ and $S_C$ is a central elementary abelian group generated by the set $\{1+\alpha (g+g^{-1})\widehat{C}\;|\; \alpha \in F,\;g\in G\setminus G_T \}$ by Lemma \ref{lemma_6}, \ref{lemma_3} and \ref{lemma_5}. 
	Furthermore, $h^{-1}(g+g^{-1})\widehat{C}h=(g+g^{-1})\widehat{C}$ for all $h\in G$, because $\overline{G}\in \Theta$. 
	
	According to Lemma \ref{lemma_main}
	\[
	\order{V_*({F}G}=
	\order{F}^{\frac{\order{G}}{2}}\cdot \frac{|V_*(F\overline{G})|}{\order{F}^{\frac{\order{G}-\order{G_T}}{4}}}=\order{F}^{\frac{\order{G}+\order{G_T}}{4}}
	\cdot \order{V_*({F}\overline{G}}.
	\]
		
\end{proof}

Let us consider some consequences of Theorem \ref{theorem_main}.

\begin{corollary}\label{corollary_every}
	Let $H$ be a group satisfying $|V_*(FH)|=n \cdot \order{F}^{\frac{|H|+|H_T|}{2}-1}$ for some $n$ and $H$ belongs to $\Theta$. Let $G=H\times E$, where $E$ is a finite elementary abelian $2$-group,  and $F$ a finite field of characteristic two.
	Then $\order{V_*(FG)}=n\cdot \order{F}^{\frac{1}{2}(|G|+|G_T|)-1}$.
\end{corollary}

\begin{proof} 
	Let $C_2$ be the cyclic group of order two
	and $c$ be the generator element of $C_2$.
	Suppose that $H$ belongs to $\Theta$ and satisfies that $|V_*(FH)|=n \cdot \order{F}^{\frac{|H|+|H_T|}{2}-1}$ for some $n$. Let $G=H\times C_2$. Using the fact that $\order{G}+\order{G_T}=2(\order{H}+\order{H_T})$ the previous theorem yields information about the order
	\[
	\order{V_*(FG)}=|F|^{\frac{\order{G}+\order{G_T}}{4}}|V_*(FH)|.
	\]
	We can compute that 
	\[
    \order{V_*(FG)}=|F|^{\frac{\order{G}+\order{G_T}}{4}} \cdot n \cdot \order{F}^{\frac{|H|+|H_T|}{2}-1}=n \cdot \order{F}^{\frac{|G|+|G_T|}{2}-1}.
    \]
    	
	We now proceed by induction. Let $E$ be a finite elementary ablelian $2$-group of order $2^m$. Choose $c\in E$ and $C=\gp{c}$. Let us denote by $N$ the factor group $G/C\cong G\times E_1$, where $\order{E_1}=2^{m-1}$. It is clear that $N$ also belongs to $\Theta$. From Theorem \ref{theorem_main} we conclude that 
	\[
	\order{V_*(FG)}=|F|^{\frac{\order{G}+\order{G_T}}{4}}|V_*(FN)|,
	\]
	where $|V_*(FN)|=n \cdot \order{F}^{\frac{|N|+|N_T|}{2}-1}$.
	Since $\order{G}+\order{G_T}=2(\order{N}+\order{N_T})$ we have that
	\[
	\order{V_*(FG)}=|F|^{\frac{\order{G}+\order{G_T}}{4}} \cdot n \cdot \order{F}^{\frac{|N|+|N_T|}{2}-1}=n \cdot \order{F}^{\frac{|G|+|G_T|}{2}-1}.
	\]
	
\end{proof} 

\begin{corollary}\label{corollary_masodrendu}
	Let $D$ and $Q$ be the dihedral and the generalized quaternion group or order $2^{m}$ $(m>2)$, respectively, $E$ be a finite elementary abelian $2$-group and $F$ a finite field of characteristic two.
	Then $\order{V_*(FG)}=n\cdot \order{F}^{\frac{1}{2}(|G|+|G_T|)-1}$, where $n=1$ if $G=D \times E$ and $n=4$ if $G=Q \times E$.
\end{corollary}

Applying Theorem \ref{theorem_main} for the semidirect product $C_4 \rtimes C_4$ we can prove the following.

\begin{lemma}\label{lemma_feldirekt}
	Let $G=(C_4 \rtimes C_4) \times E$, where $C_4 \rtimes C_4 =\gp{a,b\;\vert\; a^4=b^4=1,(a,b)=a^2}$ and $E$ is a finite elementary abelian $2$-group. Let $F$ a finite field of characteristic two.  Then $\order{V_*(FG)}=4\cdot \order{F}^{\frac{1}{2}(|G|+|G_T|)-1}$.
\end{lemma}

\begin{proof}
	Suppose that $G=C_4 \rtimes C_4$.
	Let $C$ be the group generated by $c=a^2b^2$. Since $c\in \zeta(G)[2]\setminus G_P$ and 	
	$\overline{G}=G/C\cong Q_8\in \Theta$ we have 
	\[
	|V_*(FG)|=\order{F}^{\frac{|G|+|G_T|}{4}}|V_*(FQ_8)|
	\]
	by Theorem \ref{theorem_main}.
	According to Lemma \ref{roza} $(ii)$ and the fact that $|G_T|=4$
	\[
	|V_*(FG)|=\order{F}^{\frac{|G|+|G_T|}{4}}\cdot 4\cdot \order{F}^{\frac{|G|+|G_T|}{4}-1}=4\cdot \order{F}^{\frac{|G|+|G_T|}{2}-1}.
	\]
	
	Since $(C_4 \rtimes C_4) \times E \in \Theta$, where $E$ is an elementary abelian $2$-group the proof follows from Corollary \ref{corollary_every}.
	
\end{proof}

\begin{corollary}\label{lemma_G3}
	Let $G=H_{2^n}=\gp{a,b,c\;\vert\; a^{2^{n-2}}=b^2=c^2=1,(a,b)=c, \; (a,c)=(b,c)=1,n\geq 4}$ be a group and $F$ a finite field of characteristic two. Then $\order{V_*(FG)}= 2\cdot \order{F}^{\frac{1}{2}(|G|+|G_T|)-1}$.
\end{corollary}

\begin{proof}
	Let $G'$ be the commutator subgroup of $G$ generated by $c$. It is obvious that $c\in \zeta(G)[2]\setminus G_P$ and
	$\overline{G}=G/G'\cong C_{2^{n-2}} \times C_2\in\Theta$. 
	Therefore $	|V_*(FG)|=\order{F}^{\frac{|G|+|G_T|}{4}}\big|V_*(F\overline{G})\big|$ by Theorem \ref{theorem_main}.

	According to Lemma \ref{szakacs} and the fact that $|G_T|=2|\overline{G}_T|$ we have
	\[
	|V_*(FG)|=\order{F}^{\frac{|G|+|G_T|}{4}}\cdot 2 \cdot \order{F}^{\frac{|G|+|G_T|}{4}}=2\cdot \order{F}^{\frac{1}{2}(|G|+|G_T|)-1}.
	\]
	
\end{proof}

\section{The Order of $V_*(FG)$ for basic groups of order $2^4$}

In this section we prove the conjecture that the order of $*$-unitary subgroup is divisible by $|F|^{\frac{1}{2}(\order{G}+\order{G_T})-1}$, where 
$G$ is a group of order $2^4$.
 
First, let $G$ be the semidihedral group of order $2^4$.
This group is defined by the following generators and their relations
\[
D^-_{16}=\gp{a,b\;\vert\; a^{8}=b^2=1,(a,b)=a^{2}}.
\]

We need the following two lemma.

\begin{lemma}\label{lemma_field1}
	Let $F$ be a finite field of characteristic two. The mapping $\tau : F \rightarrow F$, such that $\tau(x)=x^2+x$ is a homomorphism on the additive group of $F$, with kernel $\ker(\tau) =\{0,1\}$. 
\end{lemma}

\begin{proof}
	It is clear that $\tau(x+y)=(x+y)^2+(x+y)=x^2+y^2+x+y=\tau(x)+\tau(y)$ for every $x,y\in F$. Since $x^2+x=x(x+1)=0$ if and ony if either $x=0$ or $x=1$, the proof is complete.	
\end{proof}

For a given parameter $A\in F$ and unknowns $w_1,w_2,w_3,w_4 \in F$ let us define the following equation system
\begin{equation}\label{system}
\begin{cases}
w_1+w_2+w_3+w_4=1\\
w_1w_4+w_2w_3=A \\
w_1w_2+w_3w_4=0.
\end{cases}
\end{equation}

\begin{lemma}\label{lemma_field}
	Let $\mathbb S$ be a subset of the field $F$ which contains all the elements $A\in F$ for which the equation system \ref{system} has a solution in $F$. Then $\order{\Bbb S}=\frac{\order{F}}{2}$.
\end{lemma}

\begin{proof} 
	First, we will prove that ${\Bbb S}  \subseteq  		\im(\tau)$. Suppose that $A\in {\Bbb S}$ and 	
	$w_1,w_2,w_3,w_4\in F$ satisfy the equation system \ref{system}. Then $\tau(w_1+w_3)=(w_1+w_3)^2+(w_1+w_3)=(w_1+w_3)(1+w_1+w_3)=(w_1+w_3)(w_2+w_4)=A$. Thus for $w=w_1+w_3$ we have $\tau(w)=A$ so ${\Bbb S} \subseteq \im(\tau)$. 
	
	Assume that $\tau(w)=A$ for some $w\in F$. 
	If $w=0$, then $\tau(w)=A=0$ and $w_1=0,w_2=1,w_3=0,w_4=0$ is a solution of the equation system \ref{system}.
	Let $w_1+w_3=w\ne 0$ for some $w_1,w_3\in F$. Set $w_2=(A+w_1+ww_1)w^{-1}$ and $w_4=w_2+w+1$. It is clear that 
	$w_1+w_3+w_2+w_4=w+w+1=1$.
	Furthermore,
	$w_1w_2+w_3w_4=w_1w_2+(w_1+w)(w_2+w+1)=w_1(1+w)+A+ww_2$, because $\tau(w)=w^2+w=A$. Since $w_2=(A+w_1+ww_1)w^{-1}$ we can compute that $w_1(w+1)+ww_2+A=w_1(1+w)+(A+w_1+ww_1)+A=0$.
	Thus we have proved that $w_1w_2+w_3w_4=0$.
	Finally, 
	$A=w(w+1)=(w_1+w_3)(w_2+w_4)=w_1w_2+w_1w_4+w_2w_3+w_3w_4=w_1w_4+w_2w_3$, which shows that $\im(\tau) = {\Bbb S}$ and the proof is complete.
	
\end{proof}

Since the subgroup $N=\gp{a^2}$ is normal in $D^-_{16}$, every element of $FD^-_{16}$ can be written in the following form $x=x_1+x_2a+x_3b+x_4ab$, where $x_i \in FN$.
Let us compute
\[
\begin{split}
&(x_1+x_2a+x_3b+x_4ab)(x_1+x_2a+x_3b+x_4ab)^*=\\
&\big(x_1x_1^*+x_2x_2^*+x_3x_3^*+x_4x_4^*)+(x_2x_1^*+x_4x_3^*)a+(x_1x_2^*+x_3x_4^*)a^{7}+\\
&(x_1x_4^*+x_2x_3^*)(a+a^5)b.
\end{split}
\]

Consider the natural homomorphism of $FG$ to $F$, which is called augmentation and denoted by $\chi$. Set $w_i=\chi(x_i)$. Using the previous computations and the fact that $\zeta(G)=\gp{a^4}$ we have proved that if $xx^*\in S_{\zeta(G)}$, then $w_1+w_2+w_3+w_4=1$ and $w_1w_2+w_3w_4=0$. Therefore if $xx^*\in S_{\zeta(G)}$, then there exist $w_1,w_2,w_3,w_4\in F$ satisfying equation system \ref{system}, for some $A \in {\Bbb S}$. 

\begin{lemma}\label{lemma_semidihedral}
	Let $G=D^-_{16}$ be the semidihedral group of order $16$ and $F$ a finite field of characteristic two. Then $\order{V_*(FD^-_{16})}=2\cdot |F|^{\frac{|G|+|G_T|}{2}-1}$.
\end{lemma}

\begin{proof}
	Let $M$ be the set $\{g \in G \;|\; g^2=a^4 \}$ and $C=\zeta(G)$.
    It is clear that 
    \[
      M=\{ a^2, a^6,ab,a^3b, a^5b, a^7b \}.
    \]

    Every $*$-symmetric element of $I(C)^+$ can be written as
    \[
     \begin{split}
     1+\alpha_1 (a+a^{-1})\widehat{C}&+\alpha_2 a^2\widehat{C}+\alpha_3 ab\widehat{C}+\\ &\alpha_4 a^3b\widehat{C}+\alpha_5 b+\alpha_6 a^2b+\alpha_7 a^4b+\alpha_8 a^6b,
     \end{split}
    \] 
    where $\alpha_i\in F$ by Corollary \ref{corrolary_1}. According to Lemma \ref{lemma_5} the element $1+\alpha (a+a^{-1})\widehat{C}$ belongs to $S_{C}$ for any $\alpha \in F$.	It follows from Lemma \ref{lemma_3} that $1+\alpha g \not\in S_C$ if $g\in G_T$. 
	
	Since $\delta+\delta a^2+a \in V(FG)$ for every $\delta \in F$, an easy computation shows that
	\[
	\begin{split}
	\big(\delta+\delta a^2+a\big)\big(\delta+\delta a^2+a)^*=
	1+\delta^2(a^2+a^{-2})=1+\delta^2 a^{2}\widehat{C},
	\end{split}
	\]
	which confirm that  $\delta+\delta a^2+a \in N^*_{\Psi}$.
	Since $\eta(\alpha)=\alpha^2$ is an automorphism of $U(F)$ we can pick $\delta$ such that $\alpha_2=\delta^2$. Therefore
	$1+\alpha_2 a^{2}\widehat{C} \in S_{C}$ for every $\alpha_2\in F$.
	
	A straightforward computation shows that
	\[
	\big(\alpha (a+a^7)+b\big)\big(\alpha (a+a^7)+b\big)^*=1+\alpha^2 a^2\widehat{C}+\alpha (ab+a^3b)\widehat{C}
	\]
	for every $\alpha \in F$ so $\alpha (a+a^7)+ b\in N^*_{\Psi}$.	
	Using Lemma \ref{lemma_6} and the fact that $1+\alpha_2 a^2\widehat{C} \in S_{C}$  we have that $1+\alpha (ab+a^3b)\widehat{C}\in S_C$ for every $\alpha \in F$.
	
	We have proved that the group $N_1$ generated by the set
	\[
	\begin{split}
	\{ 1+\alpha (a+a^{-1})\widehat{C} \} \cup \{ 1+\alpha a^2\widehat{C} \} \cup \{1+ \alpha (ab+a^3b)\widehat{C} \},
	\end{split}
	\]	
	where $\alpha \in F$ is a subgroup of $S_C$ by Lemma \ref{lemma_6} and $\order{N_1}=\order{F}^{3}$ .
	
	Let $w_1+w_2a+w_3b+w_4ab\in FD^-_{16}$ be such that $w_1,w_2,w_3,w_4 \in F$ satisfy the equation system \ref{system}. We have seen that 
	\[
	(w_1+w_2a+w_3b+w_4ab)(w_1+w_2a+w_3b+w_4ab)^*=1+(w_1w_4+w_2w_3)ab\widehat{C}.
	\]
	According to Lemma \ref{lemma_field} the group $N_2$ generated by the elements $1+\alpha ab\widehat{C}$, where $\alpha \in \Bbb S$ is a subgroup of $S_C$ with order $\frac{\order{F}}{2}$. By a similar argument, $1+\alpha a^3b\widehat{C}$ belongs to $S_C$ if $\alpha \in \Bbb S$. It follows that $S_C=N_1 \times N_2$ and $\order{S_C}=\frac{\order{F}^4}{2}$

	Since $\overline{G}=G/\zeta(G) \cong D_{8}$ we have that
	$|V_*(F\overline{G})|=|F|^{\frac{3|G|}{8}}$ by Lemma \ref{roza} (i). 
	It is clear that $\frac{3|G|}{8}-3=\frac{|G_T|}{2}$.
	According to Lemma \ref{lemma_main}
	\[
	|V_*(FG)|=2\cdot \order{F}^{\frac{|G|}{2}}|F|^{\big(\frac{3|G|}{8}-3\big)-1}=2\cdot |F|^{\frac{|G|+|G_T|}{2}-1}.
	\]
	
\end{proof}

\begin{lemma}\label{lemma_modular}
	Let $G=M_{16}=\gp{a,b\;\vert\; a^8=b^2=1,(a,b)=a^4}$ be and $F$ a finite field of characteristic two.
	Then the order of $V_*(FG)$ is equal to $2\cdot |F|^{\frac{|G|+|G_T|}{2}-1}$.
\end{lemma}

\begin{proof}
	Suppose that $y \in S_{G'}$. Evidently, $y$ may be expressed as	\[
	y=1+\beta_1 (a+a^3)\widehat{G'}+\beta_2 a^2\widehat{G'}+\beta_3 \widehat{G'}+\beta_4 b\widehat{G'}+\beta_5 (a+a^3)b\widehat{G'}+\beta_6 a^2b\widehat{G'},
	\] 
	where $\beta_i \in F$.
	
    According to Lemma \ref{lemma_3} the elements $1+\beta_3 \widehat{G'}$ and $1+\beta_4 b\widehat{G'}$ do not belong to $S_{G'}$. Lemma \ref{lemma_5} shows that $\beta_1 (a+a^3)\widehat{G'}$ and $\beta_5 (a+a^3)b\widehat{G'}\in S_{G'}$ for all $\beta_1,\beta_5 \in F$.
    
    Since $\eta(\alpha)=\alpha^2$ is an automorphism of $U(F)$ we can pick $\alpha$ such that $\beta_2=\alpha^2$. Then    
    $\alpha^2+a+\alpha^2 a^2 \in V(FG)$ and we can compute that
    \[
    (\alpha^2+a+\alpha^2 a^2)(\alpha^2+a+\alpha^2 a^2)^*=1+\beta_2 a^2\widehat{G'},
    \]
    which proves that $\alpha^2+a+\alpha^2 a^2 \in N^*_{\Psi}$ and
    $1+\beta_2 a^2\widehat{G'}\in S_{G'}$ for every $\beta_2 \in F$.
    The following identity 
    \[
    \begin{split}
    \big(\alpha a^2+(1+\alpha a^2)b\big)\big(\alpha a^2+(1+\alpha a^2)b\big)^*&=1+\alpha a^2\widehat{G'}+\alpha a^2\widehat{G'}b 
    \end{split}
    \]
    shows that $\alpha a^2+(1+\alpha a^2)b \in N^*_{\Psi}$.
    Therefore $1+\alpha a^2\widehat{G'}+\alpha a^2b\widehat{G'} \in S_{G'}$ for every $\alpha \in F$.
    Since $1+\alpha a^2\widehat{G'}\in S_{G'}$ we conclude that    
    $1+\alpha a^2b\widehat{G'}$ belongs to $S_{G'}$ by Lemma \ref{lemma_6}.
    
    We have proved that $S_{G'}$ is an elementary abelian group generated by the set 
    \[
    \{ 1+\alpha_1 (a+a^3)\widehat{G'}, 1+\alpha_2 a^2\widehat{G'}, 1+\alpha_3 (a+a^3)b\widehat{G'}, 1+\alpha_4 a^2b\widehat{G'} \;|\; \alpha_i \in F \}.
    \]	
    Thus $\order{S_{G'}}=\order{F}^4$. Since $\overline{G}=G/G'\cong C_4 \times C_2$ we have that
	$|V_*(F\overline{G})|=2\cdot |F|^{5}$ by Lemma \ref{szakacs}.
	
	Finally, Lemma \ref{lemma_main} and the fact that $|G_T|=4$ shows that 
	\[
	|V_*(FG)|=2\cdot |F|^{\frac{|G|}{2}+1} =2\cdot |F|^{\frac{|G|+|G_T|}{2}-1}.
	\]
	
\end{proof}

The central product of $D_8\Y C_4$ is defined by the following generators and their relations
\[
G=\gp{a,b,c\;\vert\; a^4=b^2=c^4=1,(a,b)=a^2=c^2, \; (a,c)=(b,c)=1}.
\]

\begin{lemma}\label{lemma_DYC4}
	Let $G=D_8 \Y C_4$ be and $F$ a finite field of characteristic two. Then $\order{V_*(FG)}=|F|^{\frac{|G|+|G_T|}{2}-1}$.
\end{lemma}

\begin{proof}
	It is clear that the commutator subgroup $G'=\gp{a^2}$ and the set $\{g \in G \;|\; g^2=a^2 \}$ coincides with $\{ a,a^3,c,a^2c,bc,abc,a^2bc,a^3bc \}$.
	
	We will prove that $S_{G'}$ is generated by the set $
	\{ 1+\alpha g\widehat{G'} \;|\; g\in G\setminus{G_T} \}$
	and $\alpha \in F$. According to Lemma \ref{lemma_5} each $x\in S_{G'}$ can be written as $x= 1+\alpha_1 a\widehat{G'}+\alpha_2 bc\widehat{G'}+\alpha_3 c\widehat{G'}+\alpha_4 abc\widehat{G'}$.
	
	We first compute	
	\[
	\begin{split}
	(1+\alpha b+\alpha a)(1+\alpha b+\alpha a)^*&=1+\alpha a\widehat{G'},\\
	(1+\alpha c+\alpha a)(1+\alpha c+\alpha a)^*&=1+\alpha c\widehat{G'}+\alpha a\widehat{G'},\\
	(1+\alpha c+\alpha b)(1+\alpha a^2c+\alpha b)^*&=1+\alpha c\widehat{G'}+\alpha^2 bc\widehat{G'},\\
	(a^2c+\alpha ab+\alpha ac)(a^2c+\alpha ab+\alpha ac)^*&=1+(\alpha abc+\alpha a+\alpha^2 bc)\widehat{G'}.
	\end{split}
	\]
	
	Thus the set 
	$\{ 1+\alpha_1 a\widehat{G'}, 1+\alpha_2 c\widehat{G'}, 1+\alpha_3 bc\widehat{G'}, 1+\alpha_4 abc\widehat{G'} \;| \;  \alpha_i \in F \}$
	generates the group $S_{G'}$ by Lemma \ref{lemma_6}. As a consequence we have that $\order{S_{G'}}=\order{F}^{4}$.
	
	Since $\overline{G}=G/G'\cong C_2 \times C_2 \times C_2$  Lemma \ref{szakacs} shows that $|V_*(F\overline{G})|=|F|^{7}$. It is obvious that $|G_T|=4$ therefore 
	\[
	\frac{|V_*(F\overline{G})|}{\order{S_{G'}}}=|F|^{\frac{|G_T|}{2}-1}.
	\]
	
	According to Lemma \ref{lemma_main}
	\[
	|V_*(FG)|=\order{F}^{\frac{|G|}{2}}|F|^{\frac{|G_T|}{2}-1}=|F|^{\frac{|G|+|G_T|}{2}-1}.
	\]
	
\end{proof} 

Theorem \ref{theorem_main} in Section $2$ presents a reduction formula for $\order{V_*(FG)}$. In this Section it was shown that the order of $V_*(FG)$ can be computed using the reduction formula for any basic groups with order $2^4$.
Summarizing, we have the following theorem which confirms the conjecture on the order of $*$-unitary subgroup.

\begin{theorem}\label{theorem_G16}
	Let $F$ be a finite field of characteristic two and $G$ is a non-abelian group of order $2^4$. Then $\order{V_*(FG)}=n\cdot |F|^{\frac{1}{2}(\order{G}+\order{G_T})-1}$, where
	\begin{enumerate}
		\item[$\bullet$] $n=1$ if  $G\cong\{ D_8 \Y C_4, D_{16}, D_8 \times C_2 \}$;
		\item[$\bullet$] $n=2$ if  $G\cong\{ M_{16}, D^-_{16}, H_{16} \}$;
		\item[$\bullet$] $n=4$ if  $G\cong\{ Q_{16}, C_4 \ltimes C_4, Q_8 \times C_2 \}$.
	\end{enumerate}
\end{theorem}

Based on Theorems \ref{theorem_main} and \ref{theorem_G16} we can compute 
$\order{V_*(FG)}$ for many larger group algebras.

\section{Isomorphism Problem of Unitary Subgroups}

In this section we deal with the $*$-unitary isomorphism problem.

\begin{theorem}\label{theorem_UI}
	Let $F$ be a finite field of characteristic two and $G$ and $H$ are non-abelian $2$-groups of order at most $2^4$. Then $V_*(FG)\cong V_*(FH)$ implies that $G\cong H$.
\end{theorem}

\begin{proof}
    According to \cite{Creedon_Gildea_I} and \cite{Creedon_Gildea_II} the theorem is true for non-abelian groups of order $2^3$. The theorem is also true for groups of order $2^4$ when $F$ is the field of two elements by \cite{Balogh_Bovdi_uniter} and \cite{Bovdi_Erdei_II}.
    
    Suppose that $|F|>2$ and $\order{G}=2^4$. Theorem \ref{theorem_G16} yields that $|V_*(FG)|= |V_*(FH)|$ if and only if $G$ is either $C_4 \ltimes C_4$ or $Q_8 \times C_2$. 
    Without loss of generality we can assume that $G\cong Q_8 \times C_2$ and $H \cong C_4 \ltimes C_4$. Let $M$ be the abelian subgroup of $G$ generated by $a$ and $c$, where $Q_8=\gp{a,b\;|\;a^4=1,a^2=b^2,ba=a^3b}$ and $C_2=\gp{c\;|\;c^2=1}$. Then every element $x\in V(FG)$ can be written as $x=x_1+x_2b$, where $x_1,x_2\in FM$. We can compute that $xx^*=x_1x_1^*+x_2x_2^*+(x_1x_2+x_1x_2a^2)b$.  Similarly,
    \begin{equation}\label{power}
        x^2=x_1^2+x_2x_2^*a^2+(x_1x_2+x_1^*x_2)b.
    \end{equation}

    Thus for any $x\in V_*(FG)$ we have that $x_1x_1^*=x_2x_2^*+1$ and $x_1x_2=x_1x_2a^2$.
    Since $x$ is a unit either $x_1$ or $x_2$ must be a unit.
    
    Now suppose that $x_1$ is a unit. From the equality $x_1x_2=x_1x_2a^2$ we conclude that $x_2(1+a^2)=0$. According to Theorem $11$ in \cite{Hill_I} $x_2$ can be written as  $x_2=\alpha_0(1+a^2)+\alpha_1(1+a^2)a+\alpha_2(1+a^2)c+\alpha_4(1+a^2)ac$ for some $\alpha_i\in F$. By equation \ref{power} and the fact that $x_1+x_1^*=\beta_0(1+a^2)+\beta_1(1+a^2)a+\beta_2(1+a^2)c+\beta_4(1+a^2)ac$, $\beta_i\in F$ we conclude that $x^2=x_1^2$ and $x_1x_1^*=1$. From part $2$ of Theorem $2$ in \cite{Bovdi_Szakacs_III} it follows that $V_*(FM)\cong M \times N$, where $N$ is an elementary abelian group. Therefore either $x^2=1$ or $x^2=a^2$.
    
    Now suppose that $x_2$ is a unit. From equation $x_1x_2=x_1x_2a^2$ we conclude that $x_1(1+a^2)=0$. Therefore $x_1$ can be written as $x_1=\alpha_0(1+a^2)+\alpha_1(1+a^2)a+\alpha_2(1+a^2)c+\alpha_4(1+a^2)ac$ for some $\alpha_i\in F$ by Theorem $11$ in \cite{Hill_I}. Equation \ref{power} and $x_2x_2^*=x_1x_1^*+1=1$ and $x_1+x_1^*=0$ imply that $x^2=x_1^2+1=1$. 
    
    We have proved that if $x\in V_*(FG)$, then $x^2$ equals either $1$ or $a^2$. Thus $|V_*^2(FG)|=|\gp{a^2}|=2$.
    
    Let us consider the $*$-unitary subgroup $V_*(FH)$, where $H\cong C_4 \ltimes C_4$. Since $(C_4 \ltimes C_4)^2 \subseteq V_*^2(FH)$, we have that $|V_*(FH)|> 2$ which proofs that $V_*(FG)$ and $V_*(FH)$ are not isomorphic groups.

\end{proof}

\bibliographystyle{abbrv}
\bibliography{uniter_Modular}	

\end{document}